\documentclass[reqno]{amsart}

\usepackage[english]{babel}
\usepackage{amsthm,  latexsym, enumerate, gastex, amssymb,url,listings,enumitem}

\usepackage[mathcal]{eucal}
\usepackage{graphicx,url}

\renewcommand{\phi}[0]{\varphi}
\renewcommand{\theta}[0]{\vartheta}
\renewcommand{\epsilon}[0]{\varepsilon}

\newcommand{\N}{\text{$\mathbb{N}$}}
\newcommand{\Z}{\text{$\mathbb{Z}$}}

\newcommand{\tab}{$\theta_{a,b}$}
\newcommand{\dsum}{\displaystyle\sum}

\DeclareMathOperator{\Ap}{Ap}

\newtheorem{theorem}{Theorem}[section]
\newtheorem{lemma}[theorem]{Lemma}
\newtheorem{corollary}[theorem]{Corollary}

\theoremstyle{definition}

\newtheorem{definition}[theorem]{Definition}
\newtheorem{example}[theorem]{Example}

\theoremstyle{remark}

\numberwithin{equation}{section}

\begin{document}

\bibliographystyle{amsplain}

\author{S. Ugolini}
\address{Dipartimento di Matematica, Universit\`{a} degli studi di Trento, Via Sommarive 14, I-38123 (Italy)}
\email{s.ugolini@unitn.it}

\date{}

\title[]
{On numerical semigroups closed with respect to the action of affine maps}

\begin{abstract}
In this paper we study numerical semigroups containing a given positive integer and closed with respect to the action of an affine map. For such semigroups we find a minimal set of generators, their embedding dimension, their genus and their Frobenius number.
\end{abstract}

\maketitle

\section{Introduction}
A numerical semigroup $G$ is a subsemigroup of the semigroup of non-negative integers $(\N, +)$ containing $0$ and such that $\N \backslash G$ is finite. 
A comprehensive introduction to numerical semigroups is given in \cite{ros_ns}. Nevertheless, for the reader's convenience we  recall some basic notions, we will make use of in the current paper.

A set $S \subseteq \N$ generates a numerical semigroup $G$, namely $G = \langle S \rangle$, if and only if 
\begin{equation*}
\gcd(S) = 1,
\end{equation*}
where $\gcd(S)$ is the greatest common divisor of the elements contained in $S$. 

Any numerical semigroup $G$ has a unique finite minimal set of generators, whose cardinality is the embedding dimension $e(G)$ of $G$.

The cardinality of $\N \backslash G$ is called the genus of $G$ and is denoted by $g(G)$, while the integer
\begin{equation*}
F(G) := \max \{ x : x \in \Z \backslash G \}
\end{equation*} 
is called the Frobenius number of $G$. 

If $n \in G \backslash \{ 0 \}$, then the set 
\begin{equation*}
\Ap(G,n) = \{ s \in G: s - n \not \in G \}
\end{equation*} 
is called the Ap\'ery set of $G$ with respect to $n$.

For any $a \in \N^* := \N \backslash \{ 0 \}$ and $b \in \N$ we define the affine map
\begin{displaymath}
\begin{array}{rcl}
\theta_{a,b} : \N & \to & \N \\
x & \mapsto & ax+b.
\end{array}
\end{displaymath}

We give the following definition.

\begin{definition}\label{tab_def}
A subsemigroup $G$ of $(\N, +)$ containing $0$ is a $\theta_{a,b}$-semigroup if $\theta_{a,b}(y) \in G$ for any $y \in G \backslash \{ 0 \}$.
\end{definition}

The  problem we deal with in the paper consists in finding the smallest \tab-semigroup $G_{a,b} (c)$ containing a given integer $c \in \N \backslash \{0, 1 \}$, once two positive integers $a$ and $b$ such that $\gcd(b,c)=1$ are chosen. We notice that under such hypotheses $G_{a,b} (c)$ is a numerical semigroup, while the same does not hold if $\gcd(b,c)>1$. Indeed, if $d:=\gcd(b,c) > 1$, then all elements in $G$ are divisible by $d$ and $G$ is not co-finite.
The existence and the structure of $G_{a,b} (c)$ are dealt with in Theorem \ref{t_g_is_h}. 

In literature some special cases of \tab-semigroups have been studied. 

In \cite{ros} the authors studied  Thabit numerical semigroups, namely numerical semigroups defined for any $n \in \N^*$ as
\begin{equation*}
T(n) := \langle \left\{ 3 \cdot 2^{n+i} - 1 :  i \in \N \right\} \rangle.
\end{equation*}
Indeed, if we set $c:=3 \cdot 2^n-1$, then $T(n) = G_{2,1} (c)$. 

Also Mersenne numerical semigroups \cite{ros_m}, namely numerical semigroups defined for any $n \in \N^*$ as 
\begin{equation*}
M(n) := \langle \left\{ 2^{n+i} - 1 :  i \in \N \right\} \rangle,
\end{equation*}
are $\theta_{2,1}$-semigroups. In this case, setting $c:=2^n-1$, we have that $M(n) = G_{2,1} (c)$.

In \cite{ros_rep}, for a given integer $b \in \N \backslash \{0, 1 \}$ and a given positive integer $n$, the authors defined   
\begin{equation*}
M(b,n) := \langle \left\{ b^{n+i} - 1 :  i \in \N \right\} \rangle
\end{equation*}
as a submonoid of $(\N, +)$. If we set $c := b^n-1$, then $M(b,n) = G_{b,b-1} (c)$. We notice that this latter is not a numerical semigroup. Indeed, we have that $\gcd(b-1,c) \not = 1$. 
\subsection*{Synopsis of the paper}
The paper is organized as follows.
\begin{itemize}
\item In Section \ref{notations} we introduce some notations.
\item In Section \ref{main_results} we present (omitting the proofs) the main results of the paper (Theorem \ref{t_g_is_h}, Theorem \ref{t_m_gen}, Corollary \ref{c_e_f} and Theorem \ref{gen}). Some examples of \tab-semigroups follow.
\item Section \ref{background} and \ref{proofs} contain all the necessary background and proofs supporting the results presented in Section \ref{main_results}. In particular, Section \ref{proofs} consists of the proofs of Theorem \ref{t_g_is_h}, Theorem \ref{t_m_gen} and Theorem \ref{gen}.
\end{itemize}

\section{Definitions and notations}\label{notations}
Let $\{a, b, c \} \subseteq \N^*$. 
\begin{itemize}
\item If $y$ and $z$ are two non-negative integers such that $y < z$, then 
\begin{align*}
[y, z[ & :=  \{x \in \N: y \leq x < z \}; \\
\left[y, z \right] & :=  \{x \in \N: y \leq x \leq z \}; \\
\left[y,+\infty \right[ & :=  \{x \in \N: y \leq x  \}.
\end{align*}

\item If $k \in \N$, then we define 
\begin{displaymath}
s_k(a) := 
\begin{cases}
0 & \text{if $k=0$,}\\
\sum_{i=0}^{k-1} a^i & \text{otherwise,}
\end{cases}
\end{displaymath}
and accordingly 
\begin{displaymath}
t_k(a,b,c) := a^k c + b \cdot s_k(a).
\end{displaymath}
Moreover, we define the set
\begin{equation*}
S (a,b,c) := \lbrace t_k(a,b,c) : k \in \N \rbrace
\end{equation*}
and, for any non-negative integer $\tilde{k}$,
\begin{equation*}
S_{\tilde{k}} (a,b,c) := \lbrace t_k(a,b,c) : k \in \N \text{ and } 0 \leq k \leq \tilde{k} \rbrace.
\end{equation*}

\item We denote by $H(a,b,c)$ the semigroup generated by $S(a,b,c)$, namely 
\begin{equation*}
H(a,b,c) := \langle S(a,b,c) \rangle.
\end{equation*}
If $\gcd(b,c) = 1$, then $\gcd(c,ac+b)=1$ too. Therefore, since $\{ c, ac+b \} \subseteq S(a,b,c)$, we have that $\gcd(S(a,b,c))=1$ and $S(a,b,c)$ generates a numerical semigroup.

\item If $K$ is a non-empty finite subset of $\N$ and $\tilde{k} := \max \{k \in K \}$, then we say that a set $\{ j_i \}_{i \in K}$ of non-negative integers is $a$-reduced if 
\begin{itemize}
\item $j_i \in [0,a]$ for any $i \in K \backslash \{ 0 \}$;
\item $j_{\tilde{k}} \not = 0$; 
\item if $j_k = a$ for some $k \in K \backslash \{ 0 \}$, then $j_i=0$ for any $i \in K \backslash \{ 0 \}$ such that $i < k$.
\end{itemize} 
\item If $\{ j_i \}_{i \in K_1}$ and $\{ \tilde{j}_i \}_{i \in K_2}$ are two $a$-reduced sets of integers indexed on two subsets $K_1$ and $K_2$ of $\N^*$ such that
\begin{align*}
k_1 & :=  \max \{k \in K_1 \},\\
k_2 & :=  \max \{k \in K_2 \},
\end{align*}
then we say that
\begin{itemize}
\item $\{ j_i \}_{i \in K_1} = \{ \tilde{j}_i \}_{i \in K_2}$ if and only if $K_1 = K_2$ and $j_i = \tilde{j}_i$ for any $i \in K_1$;
\item $\{ j_i \}_{i \in K_1} \prec \{ \tilde{j}_i \}_{i \in K_2}$ if and only if $k_1 < k_2$ or $k_1 = k_2$ and $j_M < \tilde{j}_M$, where $M := \max \{ k \in K_1 : j_k \not =  \tilde{j}_k \}$. 
\end{itemize}
\end{itemize}

\section{Main results and examples}\label{main_results}
In this and the following sections $\{a, b, c \}$ is a subset of $\N^*$, where $c \geq 2$ and $\gcd(b,c) = 1$.

The following holds.
\begin{theorem}\label{t_g_is_h}
We have that $G_{a,b}(c) = H(a,b,c)$.
\end{theorem}

In the following theorem a minimal set of generators for $G_{a,b} (c)$ is provided.

\begin{theorem}\label{t_m_gen}
Let $\tilde{k} := \min \{k \in \N : s_k(a) > c - 1  \}$.
 
Then  $S_{\tilde{k}-1} (a,b,c)$ is a minimal set of generators for $G_{a,b}(c)$.
\end{theorem}

As an immediate consequence of Theorem \ref{t_m_gen} we obtain the embedding dimension of $G_{a,b} (c)$, since for each $\tilde{k} \in \N$ we have that $|S_{\tilde{k}} (a,b,c)| = \tilde{k}+1$. 

\begin{corollary}\label{c_e_f}
Let $\tilde{k} := \min \{k \in \N : s_k(a) > c - 1  \}$.
 
Then $e (G_{a,b}(c)) = \tilde{k}$.
\end{corollary}

In the following theorem we determine the Frobenius number $F(G_{a,b} (c))$ and the genus $g(G_{a,b} (c))$ of $G_{a,b} (c)$.

\begin{theorem}\label{gen}
For any $l \in [1,c-1]$ there exists and is unique an $a$-reduced set of integers $\{ j_i^{(l)} \}_{i=1}^{k_l}$, for some positive integer $k_l$, such that 
\begin{equation*}
l = \sum_{i=1}^{k_l} j_i^{(l)} \cdot s_i(a).
\end{equation*}

Moreover, if  we define 
\begin{displaymath}
x_l :=
\begin{cases}
0 & \text{if $l=0$,}\\
\sum_{i=1}^{k_l} j_i^{(l)} \cdot t_i (a,b,c) & \text{if $l \in [1,c-1]$,}
\end{cases}
\end{displaymath}
then the following hold:
\begin{enumerate}
\item $x_l  =  \min \{x \in G_{a,b} (c) : x \equiv b l \pmod{c}  \}$;
\item $\Ap (G_{a,b} (c),c)  = \{x_l : l \in [0,c-1] \}$;
\item $F(G_{a,b} (c))  = x_{c-1} - c$;
\item $g(G_{a,b} (c))  = \frac{1}{c} \cdot \sum_{l=1}^{c-1} x_l - \frac{c-1}{2}$.
\end{enumerate}
\end{theorem}

As a by-product of Theorem \ref{gen} we get the following membership criterion: if $n \in \N$ and $n \equiv x_l \pmod{c}$ for some $l \in [0,c-1]$, then $n \in G_{a,b} (c)$ if and only if $n \geq x_l$.

\begin{example}\label{exm_1}
In this example we study the semigroup $G_{3,1} (3)$. 

Adopting the notations introduced above we have that
\begin{equation*}
a = 3, \quad b = 1, \quad c = 3. 
\end{equation*}
Moreover,
\begin{align*}
x_0 & =  0,\\
x_1 & =  1 \cdot t_1 (3,1,3) =  1 \cdot (3 \cdot 3+1) = 10,\\
x_2 & =  2 \cdot t_1 (3,1,3) = 20.
\end{align*}
Therefore,
\begin{align*}
F(G_{3,1} (3)) & =  17,\\
g(G_{3,1} (3)) & =  9,
\end{align*}
according to Theorem \ref{gen}.

Since
\begin{equation*}
\min \{ k \in \N : s_k(3) > 2 \} = 2,
\end{equation*}
we  have that
\begin{equation*}
G_{3,1} (3) = \langle S_1 (3,1,3) \rangle
\end{equation*}
according to Theorem \ref{t_m_gen}.

The non-negative integers smaller than $21$ belonging to $G_{3,1} (3)$ are listed in the following table (the numbers in bold are the elements of $S_{1} (3,1,3)$):

\begin{center}
\begin{tabular}{|c|c|c|c|c|c|c|c|}
\hline 
0 & \textbf{3} & 6 & 9 & 12 & 15 & 18 \\ 
\hline 
 &  &  & \textbf{10} & 13 & 16 & 19  \\ 
\hline 
 &  &  &  &  &  & 20 \\ 
\hline 
\end{tabular} 
\end{center}

\end{example}

\begin{example}\label{exm_2}
In this example we study the semigroup $G_{3,1} (5)$. 

We have that
\begin{equation*}
a = 3, \quad b = 1,  \quad c = 5.
\end{equation*}

Moreover, 
\begin{align*}
x_0 & =  0,\\
x_1 & =  1 \cdot t_1 (3,1,5) = 1 \cdot (3 \cdot 5 + 1) = 16,\\
x_2 & =  2 \cdot t_1 (3,1,5) = 32,\\
x_3 & =  3 \cdot t_1 (3,1,5) = 48,\\
x_4 & =  1 \cdot t_2 (3,1,5) = 3^2 \cdot 5+4 = 49.
\end{align*}
Therefore, 
\begin{align*}
F(G_{3,1} (5)) &  =  44,\\
g(G_{3,1} (5)) & =  27,
\end{align*}
according to Theorem \ref{gen}.

Since
\begin{equation*}
\min \{ k \in \N : s_k(3) > 4 \} = 3,
\end{equation*}
we  have that
\begin{equation*}
G_{3,1} (5) = \langle S_2 (3,1,5) \rangle
\end{equation*}
according to Theorem \ref{t_m_gen}.

The non-negative integers smaller than $50$ belonging to $G_{3,1} (5)$ are listed in the following table (the numbers in bold are the elements of $S_{2} (3,1,5)$):

\begin{center}
\begin{tabular}{|c|c|c|c|c|c|c|c|c|c|}
\hline 
0 & \textbf{5} & 10 & 15 & 20 & 25 & 30 & 35 & 40 & 45 \\ 
\hline 
 &  &  & \textbf{16} & 21 & 26 & 31 & 36 & 41 & 46 \\ 
\hline 
 &  &  &  &  &  & 32 & 37 & 42 & 47 \\ 
\hline 
 &  &  &  &  &  &  &  &  & 48 \\ 
\hline 
 &  &  &  &  &  &  &  &  & \textbf{49} \\ 
\hline 
\end{tabular} 
\end{center}
\end{example}

\begin{example}\label{exm_3}
In this example we study the semigroup $G_{2,3} (4)$. 

We have that
\begin{equation*}
a = 2, \quad b = 3,  \quad c = 4. 
\end{equation*}
Moreover,
\begin{align*}
x_0 & =  0,\\
x_1 & =  1 \cdot t_1 (2,3,4) = 1 \cdot (2 \cdot 4 + 3) = 11,\\
x_2 & =  2 \cdot t_1 (2,3,4) = 22,\\ 
x_3 & =  1 \cdot t_2 (2,3,4) = 1 \cdot (4 \cdot 4 + 3 \cdot 3) = 25.
\end{align*}
Therefore, 
\begin{align*}
F(G_{2,3} (4)) & = 21,\\
g(G_{2,3} (4)) & = 13,
\end{align*}
according to Theorem \ref{gen}.

Since
\begin{equation*}
\min \{ k \in \N : s_k(2) > 3 \} = 3,
\end{equation*}
we  have that
\begin{equation*}
G_{2,3} (4) = \langle S_2 (2,3,4) \rangle
\end{equation*}
according to Theorem \ref{t_m_gen}.

The non-negative integers smaller than $28$ belonging to $G_{2,3} (4)$ are listed in the following table (the numbers in bold are the elements of $S_{2} (2,3,4)$):

\begin{center}
\begin{tabular}{|c|c|c|c|c|c|c|}
\hline 
0 & \textbf{4} & 8 & 12 & 16 & 20 & 24 \\ 
\hline 
 &  &  &  & &  & \textbf{25} \\ 
\hline 
 &  &  &  &  & 22 & 26 \\ 
\hline 
 &  & \textbf{11} & 15 & 19 & 23  & 27 \\  
\hline 
\end{tabular} 
\end{center}
\end{example}

\section{Background} \label{background}
In this section we prove some technical lemmas which we will repeatedly use in Section \ref{proofs}. 

\begin{lemma}\label{l_pre_0}
We have that $S(a,b,c) \subseteq G_{a,b} (c)$.
\end{lemma}
\begin{proof}
We prove by induction on $k \in \N$ that any $t_k (a,b,c)$ belongs to $G_{a,b}(c)$.

If $k = 0$, then $t_0(a,b,c) = c \in G_{a,b} (c)$.

Suppose now that $t_k (a,b,c) \in G_{a,b} (c)$ for some non-negative integer $k$. Then
\begin{equation*}
t_{k+1} (a,b,c) = a \cdot t_k(a,b,c) + b = \theta_{a,b} (t_k(a,b,c)) \in G_{a,b} (c). \qedhere
\end{equation*}
\end{proof}

\begin{lemma}\label{l_pre_1}
$H(a,b,c)$ is a subsemigroup of $(\N, +)$ closed with respect to the action of the map $\theta_{a,b}$.
\end{lemma}
\begin{proof}
By definition, $H(a,b,c)$ is a subsemigroup of $(\N, +)$.

We prove that $H(a,b,c)$ is closed with respect to the action of the map \tab. 

Consider an element 
\begin{equation*}
y = \sum_{k \in K} j_k \cdot t_k(a,b,c) \in H(a,b,c),
\end{equation*} 
where $K$ is a non-empty finite subset of $\N$ and $\{ j_k \}_{k \in K}$ is a set of positive integers. Let $\tilde{k}$ be a chosen element of $K$. Then

\begin{align*}
a y + b & =  \sum_{k \in K} a j_k \cdot t_k(a,b,c) + b \\
& =  \sum_{\substack{k \in K \\ k \not = \tilde{k}} } a j_k \cdot t_k(a,b,c) + a (j_{\tilde{k}}-1) \cdot t_{\tilde{k}}(a,b,c) + a \cdot t_{\tilde{k}} (a, b, c) + b \\
& =  \sum_{\substack{k \in K \\ k \not = \tilde{k}} } a j_k \cdot t_k(a,b,c) + a (j_{\tilde{k}}-1) \cdot t_{\tilde{k}}(a,b,c) + t_{\tilde{k}+1} (a, b, c).
\end{align*}
Since this latter is a linear combination of elements in $S(a,b,c)$ with coefficients in $\N$, we conclude that $ay+b \in H(a,b,c)$.
\end{proof}

\begin{lemma}\label{l_pre_2}
For any $k \in \N$ we have that the set
\begin{equation*}
I_k (a,b,c) := \{a^kc+bi: i \in [0, s_k(a)] \}
\end{equation*}
is contained in $H(a,b,c)$.
\end{lemma}
\begin{proof}
We prove the claim by induction on $k$.

Proving the base step is trivial, since $c \in H(a,b,c)$ and
\begin{equation*}
I_0(a,b,c) = \{ c \}.
\end{equation*}

Suppose now that $I_k(a,b,c) \subseteq H(a,b,c)$ for some $k \in \N$. 

For any $r \in [0, s_k(a)[$ and any $j \in [0,a]$ we have that 
\begin{equation*}
a^{k+1} c + b (a r + j) \in H(a,b,c).
\end{equation*}
In fact,
\begin{equation*}
a^{k+1} c + b (a r + j) = (a-j) \cdot (a^k c+br ) +j \cdot (a^k c+b(r+1)),
\end{equation*}
where
\begin{align*}
\{a^k c + br, a^k c+b(r+1) \} \subseteq  I_k(a,b,c).
\end{align*}

Therefore,
\begin{equation*}
\{a^{k+1} c + b i : i \in [0,a \cdot s_k(a)] \} \subseteq H(a,b,c).
\end{equation*}

Finally, 
\begin{align*}
t_{k+1} (a, b, c) & =  \theta_{a,b} (t_k(a,b,c)) \in  H(a,b,c).
\end{align*}
Hence, $I_{k+1} (a,b,c) \subseteq H(a,b,c)$ and the inductive step is proved.
\end{proof}

\begin{lemma}\label{l_pre_3}
Let $k$ be a non-negative integer such that $s_k(a) \geq c-1$. Then
\begin{equation*}
[t_k (a,b,c), + \infty[ \subseteq H(a,b,c).
\end{equation*}
\end{lemma}
\begin{proof}
Let $y \in [t_k(a,b,c), + \infty[$. Then
\begin{equation*}
y \equiv r \pmod{c}
\end{equation*}
for some $r \in [0,c-1]$.

Since $\gcd(b,c) = 1$, there exists an integer $i \in [0,c-1] \subseteq [0, s_k(a)]$ such that
\begin{equation*}
a^k c + bi \equiv r \pmod{c}.
\end{equation*}
Therefore,
\begin{equation*}
y-a^k c - bi \equiv 0 \pmod{c},
\end{equation*}
namely
\begin{equation*}
y = a^k c + bi + cq
\end{equation*}
for some non-negative integer $q$. Since 
\begin{equation*}
\{ a^k c + bi, c \} \subseteq H(a,b,c)
\end{equation*}
according to Lemma \ref{l_pre_2}, we conclude that $y \in H(a,b,c)$ and the result follows.
\end{proof}

\begin{lemma}\label{l_pre_5}
If $k$ is a positive integer such that $s_k(a) \leq c-1$, then 
\begin{equation*}
t_k(a,b,c) \not \in \langle S_{k-1} (a,b,c) \rangle. 
\end{equation*}
\end{lemma}
\begin{proof}
Suppose by contradiction that
\begin{equation*}
\sum_{i \in K} j_i \cdot t_i(a,b,c) = t_k(a,b,c)
\end{equation*}
for some positive integers $\{ j_i \}_{i \in K}$ indexed on a non-empty set $K \subseteq [0, k-1]$. Before proceeding we define $K^*:=K \backslash \{ 0 \}$.

We distinguish three different cases.
\begin{itemize}
\item \emph{Case 1:} $\sum_{i \in K} j_i \cdot a^i \leq a^k$. We distinguish four subcases.
\begin{itemize}
\item \emph{Subcase 1:} $a = 1$. Then
\begin{equation*}
\sum_{i \in K} j_i 1^i \leq 1^k.
\end{equation*}
This latter is possible only if $|K| = 1$ and the only integer $j_i$ is equal to $1$. Therefore, $K = \{ r \}$ for some $r \in [0,k-1]$ and $j_r = 1$. Hence,
\begin{equation*}
\sum_{i \in K} j_i \cdot t_i (a,b,c) = c + b \cdot s_r(1),
\end{equation*}
while
\begin{equation*}
t_k (a,b,c) = c + b \cdot s_k (1).
\end{equation*}
Since $s_r(1) < s_k(1)$, we conclude that $\sum_{i \in K}  j_i \cdot t_i (a,b,c) \not = t_k(a,b,c)$.

\item \emph{Subcase 2:} $a > 1$ and $\sum_{i \in K^*} j_i =0$. Then $K = \{ 0 \}$ and 
\begin{equation*}
\sum_{i \in K} j_i \cdot t_i (a,b,c) = j_0 \cdot c \equiv 0 \pmod{c}.
\end{equation*}
Since 
\begin{equation*}
t_k(a,b,c) \equiv b \cdot s_k(a) \pmod{c}
\end{equation*}
and
\begin{equation*}
b \cdot s_k(a) \not \equiv 0 \pmod{c}
\end{equation*}
because $\gcd(b,c) = 1$ and $1 \leq s_k(a) \leq c-1$, we conclude that  $\sum_{i \in K}  j_i \cdot t_i (a,b,c) \not = t_k(a,b,c)$.

\item \emph{Subcase 3:} $a > 1$ and $\sum_{i \in K^*} j_i =1$. Then either $K = \{ r \}$ or $K= \{ 0, r \}$ for some positive integer $r$. In both cases, $j_r = 1$.

If $K = \{ r \}$, then 
\begin{equation*}
\sum_{i \in K} j_i \cdot t_i (a,b,c) = a^r c + b \cdot s_r(a)  < t_{k} (a,b,c).
\end{equation*}

If $K= \{ 0, r \}$, then
\begin{equation*}
\sum_{i \in K} j_i \cdot t_i (a,b,c) = j_0 \cdot c + a^r c + b \cdot s_r(a).
\end{equation*}

We notice that
\begin{displaymath}
\begin{array}{rcll}
t_k (a,b,c) & \equiv & b \cdot s_k (a)  & \pmod{c},\\
\sum_{i \in K} j_i \cdot t_i (a,b,c) & \equiv & b \cdot s_r (a)&  \pmod{c}.
\end{array}
\end{displaymath}
Since $\gcd (b,c) = 1$ and
\begin{equation*}
0 < s_k (a) - s_r(a) < c-1,
\end{equation*}
we conclude that
\begin{equation*}
t_k (a,b,c) \not \equiv \sum_{i \in K} j_i \cdot t_i (a,b,c) \pmod{c},
\end{equation*}
and consequently $t_k(a,b,c) \not = \sum_{i \in K} j_i \cdot t_i (a,b,c)$.

\item \emph{Subcase 4:} $a > 1$ and $\sum_{i \in K^*} j_i \geq 2$. Then
\begin{align*}
\sum_{i \in K} j_i \cdot t_i(a,b,c) & =  \sum_{i \in K} j_i \cdot a^i c + b \cdot \sum_{i \in K^*} j_i \cdot  s_i(a) \\
& \leq  a^k \cdot c + b \cdot  \left( \sum_{i \in K^*} j_i \cdot \frac{a^i-1}{a-1} \right) \\
& =  a^k \cdot c + \frac{b}{a-1} \cdot \left( \sum_{i \in K^*} j_i \cdot a^i - \sum_{i \in K^*} j_i \right) \\
& \leq  a^k \cdot c + \frac{b}{a-1} \cdot \left( a^k - \sum_{i \in K^*} j_i \right) \\
& \leq  a^k \cdot c + \frac{b}{a-1} \cdot \left( a^k - 2 \right)\\
& <  a^k \cdot c + b \cdot \frac{a^k-1}{a-1} = t_k (a,b,c).
\end{align*}
\end{itemize}

\item \emph{Case 2:} $\sum_{i \in K} j_i \cdot a^i > a^k$ and $\sum_{i \in K} j_i \cdot s_i(a) > s_k(a)$. Then
\begin{equation*}
\sum_{i \in K} j_i \cdot t_i(a,b,c) > t_k(a,b,c),
\end{equation*}
in contradiction with the initial assumption.

\item \emph{Case 3:} $\sum_{i \in K} j_i \cdot a^i > a^k$ and $\sum_{i \in K} j_i \cdot s_i(a) \leq s_k(a)$. Since
\begin{equation*}
\left(\sum_{i \in K} j_i a^i - a^k \right) \cdot c = b \cdot \left( s_k(a) - \sum_{i \in K} j_i s_i (a) \right)
\end{equation*}
and
\begin{equation*}
\left(\sum_{i \in K} j_i a^i - a^k \right) \cdot c > 0,
\end{equation*}
we get that
\begin{equation*}
0 < s_k(a) - \sum_{i \in K} j_i s_i (a) \leq c-1.
\end{equation*}
This latter fact implies that
\begin{equation*}
b \cdot \left( s_k(a) - \sum_{i \in K} j_i s_i (a) \right) \not \equiv 0 \pmod{c},
\end{equation*}
in contradiction with the fact that
\begin{equation*}
\left(\sum_{i \in K} j_i a^i - a^k \right) \cdot c \equiv 0 \pmod{c}.
\end{equation*}
Hence, also in this case $\sum_{i \in K} j_i \cdot t_i(a,b,c) \not = t_k(a,b,c)$. \qedhere
\end{itemize}
\end{proof}
\begin{lemma}\label{l_pre_5_a}
Let $k$ be a positive integer. 

If $x$ is a positive integer such that
\begin{equation*}
s_k(a) \leq x < s_{k+1}(a),
\end{equation*}
then there exists and is unique an $a$-reduced set of  integers $\{ j_i \}_{i=1}^k$ such that
\begin{equation*}
x = \sum_{i=1}^k j_i \cdot s_i(a).
\end{equation*}
\end{lemma}
\begin{proof}
We prove the claim by induction on $k \in \N^*$.

If $k=1$, then $s_1(a) \leq x < s_2(a) = 1 + a$. Therefore $x = j_1 \cdot s_1(a)$, where $j_1 = x$.

Suppose that $k > 1$ and $s_k (a) \leq x < s_{k+1} (a)$. Then there exist and are unique two non-negative integers $q$ and $r$ such that
\begin{displaymath}
\begin{cases}
x = q  s_k(a) + r \\
0 \leq r < s_k(a)
\end{cases}
\end{displaymath}  
and $q \leq a$. 

If $r =0$, then $q \geq 1$. The result follows setting $j_k := q$ and $j_i := 0$ for any $i < k$.

If $r > 0$, then $1 \leq q < a$ and $s_{\tilde{k}} (a) \leq r < s_{\tilde{k}+1} (a)$ for some positive integer $\tilde{k}$. By inductive hypothesis we have that 
\begin{equation*}
r = \sum_{i=1}^{\tilde{k}} j_i \cdot s_i(a)
\end{equation*}  
for some $a$-reduced set of integers $\{ j_i \}_{i=1}^{\tilde{k}}$. Therefore the result follows setting $j_k := q$ and $j_i := 0$ for any $i \in [\tilde{k}+1, k-1]$. 
\end{proof}

\begin{lemma}\label{l_pre_5_b}
If 
\begin{equation*}
x = \sum_{i \in K} j_i \cdot t_i(a,b,c),
\end{equation*}
where $\{ j_i \}_{i \in K}$ is a set of positive integers indexed on a finite subset $K$ of $\N$, then
\begin{equation*}
x = \sum_{i \in \tilde{K}} \tilde{j}_i \cdot t_i(a,b,c)
\end{equation*}
for some $a$-reduced set of integers $\{ \tilde{j}_i \}_{i \in \tilde{K}}$ indexed on a finite subset $\tilde{K}$ of $\N$.
\end{lemma}
\begin{proof}
We notice that
\begin{equation*}
a \cdot t_{i_2} (a,b,c) + t_{i_1} (a,b,c) = t_{i_2+1} (a,b,c) + a \cdot t_{i_1-1} (a,b,c)
\end{equation*}
for any choice of positive integers $i_1$ and $i_2$ such that $i_1 \leq i_2$. 

We define $l:=0$, $K(l) := K$ and  $\tilde{j}_i := j_i$ for any $i \in K(l)$.

Then we enter the following iterative procedure.
\begin{enumerate}
\item If $\{ \tilde{j}_i \}_{i \in K(l)}$ is $a$-reduced, then we break the procedure, else we define 
\begin{align*}
M(l) & :=  \max \{ i \in K(l) : \tilde{j}_i \geq a \},\\
m(l) & :=  \min \{ i \in K(l) \backslash \{ 0 \} : \tilde{j}_i \not = 0 \}.
\end{align*}  
\item We set
\begin{align*}
\tilde{j}_{m(l)-1} & :=  \tilde{j}_{m(l)-1} + a,\\
\tilde{j}_{m(l)} & :=  \tilde{j}_{m(l)} - 1,\\
\tilde{j}_{M(l)} & :=  \tilde{j}_{M(l)} - a,\\
\tilde{j}_{M(l)+1} & :=  \tilde{j}_{M(l)+1} + 1.\\
\end{align*}
\item We set
\begin{align*}
K(l+1) & :=  K(l) \cup \{ m(l)-1, M(l)+1 \},\\
l & :=  l+1,
\end{align*}
and go to step (1).
\end{enumerate} 
We notice that for each $l$ we have that
\begin{equation*}
\dsum_{i \in K(l+1)} j_i = \dsum_{i \in K(l)} j_i
\end{equation*}
and at least one of the following holds:
\begin{displaymath}
m(l+1) = m(l) - 1  \quad \text{ or } \quad \dsum_{i \in K(l+1) \backslash \{ 0 \}} j_i < \dsum_{i \in K(l) \backslash \{ 0 \}} j_i.
\end{displaymath}
In particular, when $m(l) = 1$ for some integer $l$, we have that  
\begin{displaymath}
\dsum_{i \in K(l+1) \backslash \{ 0 \}} j_i < \dsum_{i \in K(l) \backslash \{ 0 \}} j_i.
\end{displaymath}
Therefore, after some iterations the procedure breaks. 

Hence, 
\begin{equation*}
x = \sum_{i \in \tilde{K}} \tilde{j}_i \cdot t_i(a,b,c),
\end{equation*}
where $\tilde{K} := K(l)$.
\end{proof}

\begin{example}
Suppose that
\begin{equation*}
a = 2, \quad b = 3, \quad c = 4.
\end{equation*}

Let $K = \{ 1, 2, 4 \}$ and 
\begin{equation*}
j_1 = 2, \quad j_2 = 4, \quad j_4 = 3.
\end{equation*}

We have that
\begin{align*}
t_0 (2,3,4) & =  4,\\
t_1 (2,3,4) & =  11,\\
t_2 (2,3,4) & =  25,\\
t_3 (2,3,4) & =  53,\\
t_4 (2,3,4) & =  109,\\
t_5 (2,3,4) & =  221.
\end{align*}

Let 
\begin{equation*}
x = \sum_{i \in K} j_i \cdot t_i (2,3,4) = 449.
\end{equation*}

We use the iterative procedure described in the proof of Lemma \ref{l_pre_5_b} with the aim to write 
\begin{equation*}
x = \sum_{i \in \tilde{K}} \tilde{j}_i \cdot t_i (2,3,4)
\end{equation*}
for some $a$-reduced set of integers $\{ \tilde{j}_i \}_{i \in \tilde{K}}$.

We set $l:=0$, $K(0):=K$ and $\tilde{j}_i = j_i$ for any $i \in K(0)$. 

Since $\{ \tilde{j}_i \}_{i \in K(0)}$ is not $a$-reduced, we define 
\begin{equation*}
M(0) := 4, \quad m(0) := 1.
\end{equation*}

Then we set
\begin{equation*}
\tilde{j}_{0} := 2, \quad \tilde{j}_{1} := 1, \quad \tilde{j}_{2} :=4, \quad \tilde{j}_{4} :=1, \quad \tilde{j}_{5} := 1,
\end{equation*}
and
\begin{align*}
K(1) & :=  \{0,1,2,4,5 \},\\
l & :=   1.
\end{align*}

Since $\{ \tilde{j}_i \}_{i \in K(1)}$ is not $a$-reduced, we define 
\begin{equation*}
M(1) := 2, \quad m(1) := 1.
\end{equation*}

Then we set
\begin{equation*}
\tilde{j}_{0} := 4, \quad \tilde{j}_{1} := 0, \quad \tilde{j}_{2} :=2, \quad \tilde{j}_3 := 1, \quad \tilde{j}_{4} :=1, \quad \tilde{j}_{5} := 1,
\end{equation*}
and
\begin{align*}
K(2) & :=  \{0,1,2,3,4,5 \},\\
l & :=   2.
\end{align*}

We notice that $\{ \tilde{j}_i \}_{i \in K(2)}$ is $a$-reduced and define $\tilde{K} := K(2)$.

We have that 
\begin{equation*}
x = \sum_{i \in \tilde{K}} \tilde{j}_i \cdot t_i (2,3,4).
\end{equation*}
\end{example}

\begin{lemma}\label{l_pre_6}
Let $k$ be a positive integer.

If $\{ j_i \}_{i=1}^k$ is an $a$-reduced set of integers, then
\begin{equation*}
\dsum_{i=1}^k j_i \cdot t_i(a,b,c) < t_{k+1} (a,b,c).
\end{equation*}
\end{lemma}
\begin{proof}
We prove the claim by induction on $k$. 

If $k=1$, then 
\begin{equation*}
\dsum_{i=1}^1 j_i \cdot t_i(a,b,c) \leq a \cdot t_1 (a,b,c) = t_2(a,b,c) - b < t_2 (a,b,c).
\end{equation*}

Let $k > 1$. We distinguish two cases.
\begin{itemize}
\item If $j_{k} = a$, then 
\begin{equation*}
\dsum_{i=1}^k j_i \cdot t_i(a,b,c) = a \cdot t_k (a,b,c) < t_{k+1} (a,b,c).
\end{equation*}

\item If $j_k \leq a-1$, then, by inductive hypothesis, we have that
\begin{align*}
\dsum_{i=1}^k j_i \cdot t_i(a,b,c) & =  \dsum_{i=1}^{k-1} j_i \cdot t_i(a,b,c) + j_k \cdot t_k (a,b,c)\\
& <  t_{k} (a,b,c) + (a-1) \cdot t_k (a,b,c) < t_{k+1} (a,b,c). \qedhere
\end{align*}
\end{itemize}
\end{proof}

\begin{lemma}\label{l_pre_7}
Suppose that
\begin{itemize}
\item $k_1$ and $k_2$ are two positive integers such that $k_1 \leq k_2$;
\item $\{ j_i \}_{i=1}^{k_1}$ and $\{ \tilde{j}_i \}_{i=1}^{k_2}$  are two different $a$-reduced sets of integers;
\item $x$ and $y$ are two integers such that
\begin{align*}
x & =  \sum_{i=1}^{k_1} j_i \cdot t_i (a,b,c),\\
y & =  \sum_{i=1}^{k_2} \tilde{j}_i \cdot  t_i (a,b,c).
\end{align*}
\end{itemize}

The following hold.
\begin{itemize}
\item If $\{ j_i \}_{i=1}^{k_1} \prec \{ \tilde{j}_i \}_{i=1}^{k_2}$, then $x < y$.
\item If $\{ \tilde{j}_i \}_{i=1}^{k_2} \prec \{ j_i \}_{i=1}^{k_1}$, then $y < x$.
\end{itemize}
\end{lemma}
\begin{proof}
Without loss of generality we suppose that $\{ j_i \}_{i=1}^{k_1} \prec \{ \tilde{j}_i \}_{i=1}^{k_2}$.

If $k_1 < k_2$, then 
\begin{equation*}
x = \sum_{i=1}^{k_1} j_i \cdot t_i (a,b,c) < t_{k_1+1} (a,b,c)  \leq t_{k_2} (a,b,c) \leq y
\end{equation*}
according to Lemma \ref{l_pre_6}.

If $k_1 = k_2$, then we define $M:= \max \{i \in [1,k_2]: j_i \not= \tilde{j}_i \}$. 

Since $\{ j_i \}_{i=1}^{k_1} \prec \{ \tilde{j}_i \}_{i=1}^{k_2}$ we have that $j_M < \tilde{j}_M$. Therefore, 
\begin{align*}
x & =  \sum_{i=1}^{k_1} j_i \cdot  t_i (a,b,c) =  \sum_{i=1}^{k_2} j_i \cdot t_i (a,b,c)  \\
& =  \sum_{i=1}^{M-1} j_i \cdot t_i (a,b,c) + j_M \cdot t_M (a,b,c) + \sum_{i=M+1}^{k_2} j_i \cdot t_i (a,b,c) \\
& <  t_M (a,b,c) + j_M \cdot t_M (a,b,c) + \sum_{i=M+1}^{k_2} \tilde{j}_i \cdot t_i (a,b,c) \\
& \leq  \tilde{j}_M \cdot t_M (a,b,c) + \sum_{i=M+1}^{k_2} \tilde{j}_i \cdot t_i (a,b,c)\\
& \leq  \sum_{i=1}^{k_2} \tilde{j}_i \cdot  t_i (a,b,c) = y. \qedhere
\end{align*}

\end{proof}

In analogy with Lemma \ref{l_pre_6} and Lemma \ref{l_pre_7} we state two more lemmas, whose proofs (which are omitted) follow the same lines as the proofs of the two lemmas above.

\begin{lemma}\label{l_pre_8}
If $\{ j_i \}_{i=1}^k$ is an $a$-reduced set of integers for some positive integer $k$, then
\begin{equation*}
\dsum_{i=1}^k j_i \cdot s_i(a) < s_{k+1} (a).
\end{equation*}
\end{lemma}

\begin{lemma}\label{l_pre_9}
Suppose that
\begin{itemize}
\item $k_1$ and $k_2$ are two positive integers such that $k_1 \leq k_2$;
\item $\{ j_i \}_{i=1}^{k_1}$ and $\{ \tilde{j}_i \}_{i=1}^{k_2}$  are two different $a$-reduced sets of integers;
\item $x$ and $y$ are two integers such that
\begin{align*}
x & =  \sum_{i=1}^{k_1} j_i \cdot  s_i (a),\\
y & =  \sum_{i=1}^{k_2} \tilde{j}_i \cdot  s_i (a).
\end{align*}
\end{itemize}

The following hold.
\begin{itemize}
\item If $\{ j_i \}_{i=1}^{k_1} \prec \{ \tilde{j}_i \}_{i=1}^{k_2}$, then $x < y$.
\item If $\{ \tilde{j}_i \}_{i=1}^{k_2} \prec \{ j_i \}_{i=1}^{k_1}$, then $y < x$.
\end{itemize}
\end{lemma}

\section{Proofs}\label{proofs}
\subsection{Proof of Theorem \ref{t_g_is_h}}
First, we notice that $H(a,b,c) \subseteq G_{a,b} (c)$ according to Lemma \ref{l_pre_0}.

According to Lemma \ref{l_pre_1}, $ H(a,b,c)$ is a subsemigroup of $(\N, +)$ closed with respect to the action of the map $\theta_{a,b}$.

Moreover, $\N \backslash  H(a,b,c)$ is finite. In fact, according to Lemma \ref{l_pre_3} we have that
\begin{equation*}
[t_k (a,b,c), + \infty[ \subseteq H(a,b,c)
\end{equation*} 
for any positive integer $k$ such that $s_k(a) \geq c-1$.  

Hence, $H(a,b,c)$ is a numerical semigroup and $G_{a,b} (c) = H(a,b,c)$.

\subsection{Proof of Theorem \ref{t_m_gen}}
According to Lemma \ref{l_pre_5} we have that
\begin{equation*}
t_k(a,b,c) \not \in \langle S_{k-1} (a,b,c) \rangle
\end{equation*}
for any positive integer $k < \tilde{k}$ . In fact, for any such $k$ we have that $s_k(a) \leq c-1$. 
Nevertheless, 
\begin{equation*}
t_{\tilde{k}} (a,b,c)  \in \langle S_{\tilde{k}-1} (a,b,c) \rangle.
\end{equation*}
The latter assertion holds since 
\begin{displaymath}
\begin{cases}
t_{\tilde{k}} (a,b,c) - a^{\tilde{k}} c = b(q c +r) \\
0 \leq r < c
\end{cases}
\end{displaymath}
for some non-negative integers $q$ and $r$. Since
\begin{equation*}
r \leq c-1 < s_{\tilde{k}} (a),
\end{equation*}
we have that
\begin{equation*}
a^{\tilde{k}} c + b r \in G_{a,b} (c)
\end{equation*}
according to Lemma \ref{l_pre_2}. 
Therefore,
\begin{equation*}
t_{\tilde{k}} (a,b,c) \in \langle S_{\tilde{k}-1} (a,b,c) \rangle, 
\end{equation*}
namely $S_{\tilde{k}-1} (a,b,c)$ is a minimal set of generators for $G_{a,b} (c)$.

\subsection{Proof of Theorem \ref{gen}}
We prove separately the $4$ assertions.
\begin{enumerate}[leftmargin=*]
\item Let $l \in [1,c-1]$. 
According to Lemma \ref{l_pre_5_a} there exists and is unique an $a$-reduced set $\{ j_i^{(l)} \}_{i=1}^{k_l}$ such that
\begin{equation*}
l = \sum_{i=1}^{k_l} j_i^{(l)} \cdot s_i(a).
\end{equation*}

Moreover,
\begin{align*}
x_l & =  \sum_{i=1}^{k_l} j_i^{(l)} \cdot t_i(a,b,c) \\
& =  \sum_{i=1}^{k_l} j_i^{(l)} \cdot a^i c + b \cdot \sum_{i=1}^{k_l} j_i^{(l)} \cdot s_i(a)\\
& \equiv  b l \pmod{c}.
\end{align*}

Let $x \in G_{a,b} (c)$. Since 
\begin{equation*}
\{bl: l \in [0,c-1] \}
\end{equation*}
is a set of representatives of the residue classes in $\Z / c \Z$, we can say that
\begin{equation*}
x \equiv b l \pmod{c}
\end{equation*}
for some $l \in [0,c-1]$. 

If $x \not \equiv 0 \pmod{c}$, then 
\begin{equation*}
x = \sum_{i \in \tilde{K}} \tilde{j}_i \cdot t_i(a,b,c)
\end{equation*}
for some $a$-reduced set of integers $\{ \tilde{j}_i \}_{i \in \tilde{K}}$, according to Lemma \ref{l_pre_5_b}.

We distinguish two cases.

\begin{itemize}
\item If $0 \in \tilde{K}$ and $\tilde{j}_0 \not = 0$, then
\begin{align*}
x & =  \tilde{j}_0 \cdot c + \sum_{i \in \tilde{K} \backslash \{ 0 \}} \tilde{j}_i \cdot t_i(a,b,c) \\
& >   \sum_{i \in \tilde{K} \backslash \{ 0 \}} \tilde{j}_i \cdot t_i(a,b,c) \equiv b l \pmod{c}.
\end{align*}
\item If $0 \not \in \tilde{K}$ or $\tilde{j}_0 = 0$, then $\{ j_i^{(l)} \}_{i=1}^{k_l} \preceq \{ \tilde{j}_i \}_{i \in \tilde{K}}$ and $x_l \leq x$ according to Lemma \ref{l_pre_7}.

Indeed, suppose by contradiction that $\{ \tilde{j}_i \}_{i \in \tilde{K}} \prec \{ j_i^{(l)} \}_{i=1}^{k_l}$. 

We notice that 
\begin{align*}
x & \equiv b \cdot \sum_{i \in \tilde{K}} \tilde{j}_i \cdot s_i(a) \equiv b l \pmod{c},
\end{align*}
namely
\begin{equation*}
\sum_{i \in \tilde{K}} \tilde{j}_i \cdot s_i(a) \equiv l \pmod{c}.
\end{equation*} 
This latter is absurd because
\begin{equation*}
1 \leq \sum_{i \in \tilde{K}} \tilde{j}_i \cdot s_i(a) < \sum_{i = 1}^{k_l} j_i^{(l)} \cdot s_i(a) = l
\end{equation*}
according to Lemma \ref{l_pre_9}. 
\end{itemize}

\item This assertion follows from (1).

\item Since $x_i < x_{c-1}$ for any $i \in [0, c-1[$, we have that $F( G_{a,b} (c)) = x_{c-1} - c$.

\item This assertion follows from Selmer's formulas.
\end{enumerate}

\bibliography{Refs}
\end{document}